\documentclass{mathincs}
\usepackage{amsmath}
\usepackage{amsthm,amscd}
\usepackage{algorithm, algorithmic}
\usepackage{bm}

\def\cA{{\mathcal A}}
\def\cL{{\mathcal L}}
\def\cM{{\mathcal M}}
\def\cN{{\mathcal N}}
\def\cO{{\mathcal O}}
\def\cR{{\mathcal R}}
\def\cS{{\mathcal S}}

\def\C{{\mathbf C}}
\def\N{{\mathbf N}}

\def\0{\bm{0}}
\def\e{\bm{e}}

\def\NT{{\mathbf{NT}}}

\def\pd{{\partial}}

\def\ann{\operatorname{Ann}}
\def\Hom{\operatorname{Hom}}
\def\ext{\operatorname{Ext}}
\def\myspan{\operatorname{Span}}
\def\nf{\operatorname{nf}}
\def\ord{\operatorname{ord}}
\def\res{\operatorname{Res}}

\def\supp{\operatorname{supp}}

\def\groebner{Gr\"{o}bner }

\def\bl{\bullet}
\def\ag#1{{\langle #1 \rangle}}

 \newtheorem{theorem}{Theorem}[section]
 
 \newtheorem{lem}[theorem]{Lemma}
 
 \theoremstyle{definition}
 \newtheorem{definition}[theorem]{Definition}
 \newtheorem{example}[theorem]{Example}
 \theoremstyle{remark}
 
 \numberwithin{equation}{section}

\begin{document}
%
%
%
\title[An algorithm for computing Grothendieck local residues]{An algorithm for computing Grothendieck local\\
residues II --- general case ---}

\author[K. Ohara]{Katsuyoshi Ohara}

\address{%
Faculty of Mathematics and Physics\\
Kanazawa University\\
Kakuma-machi, Kanazawa 920-1192, Japan}
\email{ohara@se.kanazawa-u.ac.jp}

\thanks{
This work has been partly supported by JSPS KAKENHI Grant Numbers JP17K05292, JP18K03320 
and by the Research Institute for Mathematical Sciences, a Joint Usage/Research Center located in Kyoto University.
}
\author[S. Tajima]{Shinichi Tajima}
\address{
Graduate School of Science and Technology\\
Niigata University\\
8050 Ikarashi 2 no-cho, Nishi-ku, Niigata 950-2181, Japan}
\email{tajima@emeritus.niigata-u.ac.jp}

\date{January 1, 2004}

\begin{abstract}
Grothendieck local residue is considered in the context of symbolic
computation.  Based on the theory of holonomic $D$-modules, an effective
method is proposed for computing Grothendieck local residues.  The key
is the notion of Noether operator associated to a local cohomology
class.  The resulting algorithm and an implementation are described with
illustrations.
\end{abstract}

\maketitle
\section{Introduction}

In this paper, we consider Grothendieck local residues from the point of
view of computational algebraic analysis.  In a previous
paper~\cite{OharaTajima}, we considered the so-called shape basis cases
and gave algorithm for computing Grothendieck local residues.  In the
present paper, we adopt the same framework given in~\cite{OharaTajima}
and consider the general cases.  Main ingredients of our approach are
local cohomology, Noether operators and holonomic $D$-modules.
More precisely, we consider holonomic $D$-modules defined by
annihilating left ideals in a Weyl algebra of a zero-dimensional local
cohomology class, and partial differential operators associated to
the Grothendieck local residue mapping
\[
\varphi(x) \mapsto
\res_\beta(\frac{\varphi\, dx}{f_1\cdots f_n})=
\left(\frac{1}{2\pi \sqrt{-1}}\right)^n
\int_{\Gamma(\beta)}
\frac{\varphi(x)\, dx_1\wedge\cdots\wedge dx_n}{f_1(x)\cdots f_n(x)}.
\]

We compute a Noether operator, partial differential operator that
describe or represent Grothendieck local residue mapping, by solving a
system of partial differential equations in the Weyl algebra.

The paper is organized as follows: In Section~2, we briefly recall the
notion of local residues and algebraic local cohomology groups.  In
Section~3, we describe theory of Noether operators associated to a
primary ideal in terms of $D$-modules.  In Section~4, we give a
recursive method for finding annihilators of an algebraic local
cohomology class that avoids the use of \groebner bases on the Weyl
algebra.  In Section~5, we propose algorithms for determining Noether
operator, associated to a local cohomology class, which represents the
local residue mapping.  The resulting algorithms have been implemented
in a computer algebra system Risa/Asir~\cite{N2}.  An example by
Risa/Asir will be shown in Section~6.

\section{Algebraic local cohomology groups}\label{sec:3}

Let $K$ be a subfield of $\C$ and denote $K[x]=K[x_1, \ldots, x_n]$.
We suppose that a polynomial sequence $F=\{f_1, \ldots, f_n\}\subset K[x]$ is regular.
The polynomial ideal $I$ generated by $F$ is zero-dimensional
and the zero set $V_\C(I)=\{x \in \C^n \mid g(x)=0, \ \forall g\in I\}$ consists of finite number of points.

We introduce the $n$-th {\it algebraic local cohomology group} with
support on $Z=V_\C(I)$ by
\[
H_{[Z]}^n(K[x]) = \lim_{k\to \infty} \ext_{K[x]}^n (K[x]/(\sqrt{I})^k, K[x]).
\]
The algebraic local cohomology group $H_{[Z]}^n(K[x])$ can be regarded
as a collection of equivalent classes of rational functions whose
denominator has zero only on $Z$.  Here the equivalence is given by
cutting holomorphic parts of rational functions in a cohomological way.
Since the local residue can be described in terms of local
cohomology~\cite{Ha,TN2}, we have
$\res_\beta(\frac{\varphi dx}{f_1\cdots f_n})=\res_\beta(\varphi \sigma_F dx)$
where $\sigma_F=\left[\frac{1}{f_1 \cdots f_n}\right]\in H_{[Z]}^n(K[x])$
and $dx=dx_1\wedge\cdots\wedge dx_n$.

Algorithms for computing primary decomposition
$I = \bigcap_{\lambda=1}^\ell I_\lambda$ had been established~(e.g. see~\cite{GTZ,N}).
Suppose that, on the ring $K[x]$, a primary decomposition algorithm can be executed.
Since $I$ is zero-dimensional, the associated prime
$\sqrt{I_\lambda}$ is also zero-dimensional thus maximal.
In other words, $K_\lambda = K[x]/\sqrt{I_\lambda}$ is a field.

According to the primary decomposition, the zero
set is written as a union of irreducible affine varieties:
$Z=\bigcup_{\lambda=1}^\ell Z_\lambda$, where
$Z_\lambda=V_\C(\sqrt{I_\lambda})$.
Then $H_{[Z]}^n(K[x])$ can be decomposed to a direct sum
\[
H_{[Z]}^n(K[x]) = H_{[Z_1]}^n(K[x])\oplus \cdots \oplus H_{[Z_\lambda]}^n(K[x]) \oplus \cdots \oplus H_{[Z_\ell]}^n(K[x]).
\]
Therefore an algebraic local cohomology class $\sigma_F\in H_{[Z]}^n(K[x])$ has a unique
decomposition
\[
\sigma_F = \sigma_{F, 1}+\cdots+\sigma_{F, \lambda}+\cdots+\sigma_{F, \ell},
\]
where $\sigma_{F, \lambda} \in H_{[Z_\lambda]}^n(K[x])$.
Note that $\supp\, (\sigma_{F, \lambda})\subset Z_\lambda$.
The decomposition above is a kind of partial fractional expansion of
$\frac{1}{f_1 \cdots f_n}$ in terms of local cohomology.

Let $\beta \in Z_\lambda$ and $\varphi(x) \in \cO(U)$ where $U$ is a
small neighborhood of $\beta$.  
Since $\sigma_{F, j}$ vanishes on $U$ for $j\ne \lambda$, we have 
$\res_\beta(\varphi \sigma_F dx)=\res_\beta(\varphi\, \sigma_{F, \lambda}dx)$
for $\beta\in Z_\lambda$.  


Let $D_n=K\ag{x_1, \ldots, x_n, \pd_1, \ldots, \pd_n}$ be the ring of
linear partial differential operators with polynomial coefficients over $K$.
Here we used the symbol $\pd_i = \frac{\pd}{\pd x_i}$.
The noncommutative ring $D_n$ is called the Weyl algebra in $n$ variables.

Let $P$ be a zero-dimensional prime ideal generated by $H=\{h_1,\ldots,h_n\}$ and
put $J_H=\left(\det\left[\frac{\pd h_j}{\pd x_j}\right]_{ij} \bmod P\right)$.
An algebraic local cohomology class
\[
\delta_{Y} = \left[\frac{J_H}{h_1h_2\cdots h_n}\right] \in H_{[Y]}^n(K[x])
\]
is called the {\it delta function} with the support $Y=V_\C(P)$.

Recall that an algebraic local cohomology group has a structure of
$D_n$-module.
It follows from $H_{[Y]}^n(K[x])=D_n\bl \delta_{Y}$~(e.g.~\cite{OharaTajima} Lemma~3.2)
that
there exists a linear differential operator $T_{F, \lambda}\in D_n$ such that
$\sigma_{F, \lambda} = T_{F, \lambda}^* \bl \delta_{Z_\lambda}$ where 
$T_{F, \lambda}^*$ is the formal adjoint of $T_{F, \lambda}$.  
Therefore
\begin{align*}
\res_\beta(\frac{\varphi dx}{f_1\cdots f_n})
&=
\res_\beta(\varphi \sigma_F dx)
\\
&=
\res_\beta(\varphi\sigma_{F,\lambda} dx)
\\
&=
\res_\beta(\varphi\cdot (T_{F,\lambda}^*\bl \delta_{Z_\lambda}) dx)
\\
&=
\res_\beta((T_{F,\lambda}\bl \varphi)\cdot \delta_{Z_\lambda}dx)
\\
&=
(T_{F,\lambda}\bl \varphi)|_{x=\beta}.
\end{align*}
That is, the mapping $\varphi \mapsto \res_\beta(\varphi\sigma_Fdx)$ is
determined by the differential operator $T_{F,\lambda}$.
Since the set $\{(T_{F, \lambda}, Z_\lambda) \mid \lambda=1, 2, \ldots,
\ell\}$ gives the Grothendieck local residue mappings, the local residue
of any meromorphic $n$-forms can be evaluated by differential operators
$T_{F,\lambda}$.

\section{Noether differential operators}

In 1960's, L. Ehrenpreis studied systems of linear partial differential
equations with constant coefficients.  In his theory, some differential
operators with polynomial coefficients play an important role~\cite{E}.
Today, these operators are called Noether operators (e.g.~\cite{Pa}).
They are very useful to handle local cohomology classes.  We give a
modern definition of Noether operators in terms of $D$-modules.
Throughout this section, we suppose that $J$ is a zero-dimensional
primary ideal over $K[x]=K[x_1, \ldots, x_n]$.
Set $M_J=D_n/D_nJ$ and $M_{\sqrt{J}}=D_n/D_n\sqrt{J}$.


\begin{definition}
The homomorphisms 
$\Hom_{D_n}(M_J,M_{\sqrt{J}})$ \ between left 
$D_n$-modules is called the {\it Noether space} of $J$.
\end{definition}

\begin{definition}
The ratio $m=\dim_K (K[x]/J) / \dim_K (K[x]/\sqrt{J})$
of dimensions of vector spaces is called the {\it multiplicity} of the ideal $J$.
\end{definition}

\begin{lem}
There exist $\rho_1, \rho_2, \ldots, \rho_m \in \Hom_{D_n}(M_J,M_{\sqrt{J}})$ 
such that any $\rho\in \Hom_{D_n}(M_J,M_{\sqrt{J}})$ can be written as
\[
\rho = \rho_1 c_1 + \rho_2 c_2 + \cdots + \rho_m c_m
\]
by unique polynomials $c_1, c_2, \ldots, c_m\in K[x]/\sqrt{J}$.
\end{lem}

Notice that the Noether space has a structure of a {\it right} $K[x]/\sqrt{J}$-module,
that is, the Noether space is an $m$-dimensional vector space over $K[x]/\sqrt{J}$.

Since $1\in M_J$ is an equivalent class of differential operators,
a representative $R_i$ of the image $\rho_i (1)$ is a differential operator, $i=1,2,\ldots,m$.
We call the set $\mathcal{R}=\{R_1, R_2, \ldots, R_m\}$ {\it Noether operator basis}.

Let $L_i$ denotes the formal adjoint of $R_i$.  Then we have the following theorem.
\begin{theorem}
The differential operators $L_1,L_2,\ldots,L_m$ satisfy
\begin{equation}\label{cond:nop}
J=\{h \in K[x] \mid L_1\bl h, L_2\bl h, \ldots, L_m\bl h\in \sqrt{J}\}.
\end{equation}
\end{theorem}
Note that the differential operators above coincide with {\it Noether operators} introduced by V. P. Palamodov~\cite{Pa}.

We describe a recursive method for computing Noether operators.
We can suppose that $\ord(L_1)\le \ord(L_2)\le \cdots \le \ord(L_m)$ without loss of generality.
Put $\NT = \bigoplus_{i=1}^m (K[x]/\sqrt{J})\, L_i$ and 
$\NT^{(k)} = \{ L \in \NT \mid \ord(L) \le j \}$ 
for each natural number $k$.
The recursion will be executed until $\dim_{K[x]/\sqrt{J}} \NT^{(r)} = m$. 

At first, we have $L_1=1$ because $\NT^{(0)}=K[x]/\sqrt{J}$.

Second, a first order Noether operator can be written as
$L = \sum_{i=1}^n a_i\pd_i$ where $a_i \in K[x]/\sqrt{J}$ 
because zeroth order terms are reduced.
Let $J=\ag{g_1,g_2,\ldots,g_\mu}$ and $\sqrt{J}=\ag{p_1,p_2,\ldots,p_\eta}$.
From the condition~$(\ref{cond:nop})$, for any $1\le j\le \mu$, there exist polynomials 
$b_{j1}, b_{j2}, \ldots, b_{j\eta}$ such that
$\sum_{i=1}^n a_i\pd_i\bl g_j + \sum_{k=1}^\eta b_{jk} p_k = 0$.
The system of relations above can be rewritten, as an equation on a $K[x]$-module $K[x]^\mu$, as follows:
\begin{equation}\label{eq:3.2}
\sum_{i=1}^n a_i \left(\sum_{j=1}^\mu \frac{\pd g_j}{\pd x_i} \e_j\right)
 + \sum_{j=1}^\mu \sum_{k=1}^\eta b_{jk} p_k \e_j = \0.
\end{equation}
Here $\e_j$ is the $j$-th unit vector.
Let
\[
\cN^{(1)} = \left\{ \left. \sum_{j=1}^\mu \frac{\pd g_j}{\pd x_i} \e_j \ \right| \ i = 1,2, \ldots, n \right\}
\cup 
\{ p_k \e_j \mid 1\le j\le \mu,\ 1\le k \le \mu \}.
\]
It can be regarded as a finite subset of the $K[x]$-module $K[x]^{\mu}$.
The solutions $(a_i; b_{jk})$ of the equation $(\ref{eq:3.2})$ form the syzygy module 
with respect to $\cN^{(1)}$.
Each solution corresponds to a first order Noether operator $L=\sum_{i=1}^n a_i \pd_i$
if $(a_i) \ne \0$.
Using a \groebner basis of the syzygy module, we can determine a basis of $\NT^{(1)}$.

Next we discuss construction of a basis of $\NT^{(r)}$.
Let $\{L_1,L_2,\ldots,L_q\}$ be a basis of $\NT^{(r-1)}$ and
let $L=\sum_{1\le |\alpha|\le r} a_\alpha\pd^\alpha \in \NT^{(r)}$ 
where the multi-index $\alpha$ runs over $\N_0^n$ and $a_\alpha \in K[x]/\sqrt{J}$.
The commutator $[L,f] = Lf-fL$ is also an element of $\NT^{(r-1)}$ if $f\in K[x]$.
Thus \\
$[L,x_1], [L,x_2], \ldots, [L,x_n] \in \NT^{(r-1)}$.

\begin{lem}
A differential operator $P\in D_n$ satisfies a relation $P\bl J \subset \sqrt{J}$ if and only if
\[
\{P\bl g_1,P\bl g_2,\ldots,P\bl g_\mu\} \subset \sqrt{J} \quad \text{and}\quad [P,x_i]\bl J \subset \sqrt{J} \quad (i=1,2,\ldots,n).
\]
\end{lem}

From the lemma above, we consider a system of membership problems
$L\bl g_i \in \sqrt{J}$ and $[L,x_j] \in \NT^{(r-1)}$.
That is, there exist polynomials $b_{ih}$, $b_{kj}'$, $b_{\beta h}'' \in K[x]$ such that 
\begin{eqnarray}
\label{eq:3.3a}
\sum_{1\le |\alpha|\le r} a_\alpha (\pd^\alpha \bl g_i)
 + \sum_{h=1}^\eta b_{ih} p_h &=& 0, \qquad (1\le i\le \mu)
\\
\label{eq:3.3b}
\sum_{1\le |\alpha|\le r} a_\alpha ([\pd^\alpha, x_k])
+ \sum_{j=1}^q b_{kj}' L_j
+ \sum_{1\le |\beta|< r} \sum_{h=1}^\eta b_{\beta h}''p_h\pd^\beta &=& 0. \qquad (1\le k\le n)
\end{eqnarray}
Here the commutator $[\pd^\alpha,x_k]$ and $L_j$ can be expressed as
$\sum_{0\le |\beta| < r} c_{\alpha k\beta} \pd^\beta$ and
$\sum_{0\le |\beta| < r} d_{j \beta} \pd^\beta$ respectively.
Recall $L_1=1$.
Removing zeroth order terms, the equation $(\ref{eq:3.3b})$ can be rewritten as
\[
\sum_{1\le |\beta| < r} 
\left(
\sum_{1\le |\alpha|\le r} a_\alpha 
c_{\alpha k\beta} 
+ \sum_{j=2}^q b_{kj}' 
d_{j \beta} 
+ 
\sum_{h=1}^\eta b_{\beta h}''p_h
\right)\pd^\beta
 = 0. \qquad (1\le k\le n)
\]
Thus
\begin{equation}\label{eq:3.3c}
\sum_{1\le |\alpha|\le r} a_\alpha 
c_{\alpha k\beta} 
+ \sum_{j=2}^q b_{kj}' 
d_{j \beta} 
+ 
\sum_{h=1}^\eta b_{\beta h}''p_h
 = 0. \qquad (1\le k\le n, 1\le |\beta| < r)
\end{equation}
Since the set $\{ \beta \in \N_0^n \mid 1\le |\beta| < r\}$ has $N:=\binom{n+r-1}{n}-1$ elements,
there exists a bijection $\varphi: \{1,2,\ldots,n\} \times \{ \beta \in \N_0^n \mid 1\le |\beta| < r\} \to \{\mu+1,\mu+2,\ldots,\mu+nN\}$.
The system $(\ref{eq:3.3a}),(\ref{eq:3.3c})$ can be rewritten, as single
equation on $K[x]^{\mu + n N}$, as follows: 
\begin{eqnarray}
\nonumber
&&\sum_{1\le |\alpha|\le r} a_\alpha 
\biggl\{
\sum_{i=1}^\mu(\pd^\alpha \bl g_i) \e_i + 
\sum_{k=1}^n \sum_{1\le |\beta|<r} c_{\alpha k\beta}  \e_{\varphi(k,\beta)}
\biggr\}
\\
&&\qquad
+ \sum_{j=1}^q \sum_{k=1}^n b_{kj}' 
\biggl\{\sum_{1\le |\beta|<r} d_{j \beta}  \e_{\varphi(k,\beta)}
\biggr\}
+ 
\sum_{h=1}^\eta \sum_{t=1}^{\mu+nN} f_{th} p_h \e_t
= \0,
\label{eq:3.6}
\end{eqnarray}
where $\e_t$ is the $t$-th fundamental vector of dimension $\mu+nN$.
The equation $(\ref{eq:3.6})$ can be solved by \groebner basis computation 
of the syzygy module with respect to a finite set
\begin{eqnarray}
&&
\cN^{(r)} = \left\{
\left.
\sum_{i=1}^\mu(\pd^\alpha \bl g_i) \e_i + 
\sum_{k=1}^n \sum_{1\le |\beta|<r} c_{\alpha k\beta} \e_{\varphi(k,\beta)}
\ \right| \ 1\le |\alpha|\le r
\right\}
\nonumber
\\
&&
\qquad\qquad
\cup
\left\{
\left.
\sum_{1\le |\beta|<r} d_{j \beta} \e_{\varphi(k,\beta)} 
\ 
\right|
\ 
1\le j\le q,
1\le k\le n
\right\}
\nonumber
\\
&&
\qquad\qquad
\cup\ 
\{ p_h \e_t \mid 1\le h\le \eta, \ 1\le t \le \mu+nN \}.
\label{eq:3.7}
\end{eqnarray}
A solution $(a_\alpha; b_{kj}'; f_{th})$ gives an $r$-th order Noether operator $L=\sum_{\alpha} a_\alpha \pd^\alpha$ if there exists $\gamma\in \N_0^n$ such that $a_\gamma \ne 0$ and $|\gamma|=r$.
Hence a basis of $\NT^{(r)}$ can be determined.

\begin{algorithm}[H]
\caption{a basis of $\NT^{(1)}$}
\label{alg:noether1}
\begin{algorithmic}
\REQUIRE $J=\ag{g_1, g_2,\ldots, g_\mu}$, $\sqrt{J}=\ag{p_1,p_2,\ldots,p_\eta}$
\ENSURE $\cL^{(1)}$: a basis of $\NT^{(1)}$.
\STATE $E \gets \text{(the unit matrix of size $\mu$)}$
\STATE $J \gets \left[\frac{\pd g_j}{\pd x_i}\right]_{ji}$: Jacobi matrix of size $(\mu,n)$
\STATE $M \gets [ J \, | \, p_1 E \, | \, \cdots \, | \, p_\eta E ]$: combined matrix
\STATE $S \gets \text{(\groebner basis of the syzygy module of columns vectors of $M$)}$
\STATE $\cL^{(1)}\gets \{1\}$
\FORALL{$(a_i; b_{jk})\in S$}
\IF{$(a_1,\ldots,a_n)\ne \0$}
\STATE $\cL^{(1)} \gets \cL^{(1)} \cup \{\sum_{i=1}^n a_i \pd_i \}$
\ENDIF
\ENDFOR
\RETURN $\cL$
\end{algorithmic}
\end{algorithm}

\begin{algorithm}[H]
\caption{a basis of $\NT^{(r)}$}
\label{alg:noether2}
\begin{algorithmic}
\REQUIRE $r\ge 2$, $J$, $\sqrt{J}$, $\cL^{(r-1)}$: a basis of $\NT^{(r-1)}$
\ENSURE $\cL^{(r)}$: a basis of $\NT^{(r)}$.
\STATE $\cN^{(r)} \gets$ (the right hand side of the equation $(\ref{eq:3.7})$ by $J$, $\sqrt{J}$, $\cL^{(r-1)}$)
\STATE $S \gets$ (\groebner basis of the syzygy module with respect to $\cN^{(r)}$)
\STATE $\cL^{(r)} \gets \cL^{(r-1)}$
\FORALL{$(a_\alpha; b_{kj}';f_{th})\in S$}
\STATE $L \gets \sum_{1\le |\alpha|\le r} a_\alpha\pd^{\alpha}$
\IF{$\ord(L)=r$}
\STATE $\cL^{(r)} \gets \cL^{(r)} \cup \{L \}$
\ENDIF
\ENDFOR
\RETURN $\cL^{(r)}$
\end{algorithmic}
\end{algorithm}

\begin{algorithm}[H]
\caption{a basis of $\NT$}
\label{alg:noether3}
\begin{algorithmic}
\REQUIRE $J=\ag{g_1, g_2,\ldots, g_\mu}$, $\sqrt{J}=\ag{p_1,p_2,\ldots,p_\eta}$
\ENSURE $\cL$: a basis of $\NT$.
\STATE $m\gets \text{(the multiplicity of $J$)}$
\STATE $\cL \gets$ (the result by Algorithm~\ref{alg:noether1} from $J$ and $\sqrt{J}$)
\STATE $r\gets 2$
\WHILE{$|\cL|<m$}
\STATE $\cL \gets$ (the result by Algorithm~\ref{alg:noether2} from $r$, $J$, $\sqrt{J}$, and $\cL$)
\STATE $r\gets r+1$
\ENDWHILE
\RETURN $\cL$
\end{algorithmic}
\end{algorithm}

\begin{example}\label{ex:3.6}
Let $f_1=(x^2-2)(x^4-4x^2-y^4-5y^3-9y^2-7y+2)$ and $f_2=x^4-4x^2-y^3-3y^2-3y+3$.
Then $J=\ag{f_1,f_2}$ is a primary ideal with multiplicity 11.
Using Algorithm~\ref{alg:noether3}, Noether operators $\cL=\{L_1,L_2,\ldots,L_{11}\}$ can be computed as follows:
\begin{eqnarray*}
&&
L_1=1,  \qquad 
L_2={\pd_1}, \qquad 
L_3={\pd_2}, \qquad 
L_4={\pd_2}^2, \qquad 
L_5={\pd_1}{\pd_2} , \qquad 
L_6=3 {\pd_1}^2 + 8 {\pd_2}^3, \\
&&
L_7={\pd_1}{\pd_2}^2 , \qquad 
L_8=3 {\pd_1}^2{\pd_2} +2 {\pd_2}^4, \qquad 
L_9=2 {\pd_1}^3-3 x {\pd_1}^2+16 {\pd_1}{\pd_2}^3 , \\
&&
L_{10}=15 {\pd_1}^2{\pd_2}^2  + 4 {\pd_2}^5, \qquad 
L_{11}= 30 {\pd_1}^4 - 60 x {\pd_1}^3 + 480 {\pd_1}^2{\pd_2}^3 + 135{\pd_1}^2 + 64 {\pd_2}^6.
\end{eqnarray*}
\end{example}

\section{Annihilators}

Let $\ann_{D_n}(\sigma_F)$ denote the annihilating
left ideal in the Weyl algebra $D_n$
of the algebraic local cohomology class
$\sigma_F=\left[\frac{1}{f_1 \cdots f_n}\right]$ 
where $F=\{f_1,f_2, \ldots, f_n\}$ is a regular sequence.
In general, the annihilating left ideal can be computed by using
restriction algorithm that involves \groebner basis computation
in Weyl algebras~\cite{O, OakuT}.
However the cost of computation of the algorithm that outputs \groebner basis of the ideal 
$\ann_{D_n}(\sigma_F)$ is quit high.

Since our approach requires only generators of the annihilating ideal,
we adopt an alternative method introduced in \cite{TN2} for constructing generators 
of $\ann_{D_n}(\sigma_F)$ for avoiding  \groebner bases computation in the Weyl algebra.
In this section, we give a recursive method.

Let $\cA^{(k)}$ denote the set of $k$-th or lower order annihilators of $\sigma_F$.
Since $K[x]\subset D_n$, the polynomials can be regarded as zeroth order 
differential operators.
From the definition of the algebraic local cohomology class $\sigma_F$,
a polynomial $g(x)$ satisfies $g(x)\bl\sigma_F = 0$ if and only if $g(x)
\in I$.
Thus the polynomial ideal $I$ coincides with $\cA^{(0)}$.

Let $P\in \cA^{(k)}$ and $f\in F$.
The commutator $[P,f]$ is also an annihilator in $\cA^{(k-1)}$.
Inversely, the annihilator $P$ can be determined from commutators.
\begin{lem}\label{lemma:4.1}
If a differential operator $P'$ satisfies $[P',f_1], [P',f_2], \ldots, [P',f_n]\in \ann_{D_n}(\sigma_F)$, then there exists a polynomial $c(x)$ such that $P'+c \in \ann_{D_n}(\sigma_F)$.
\end{lem}

\begin{proof}
We can suppose $P'\bl \sigma_F \ne 0$ without loss of generality.
Since $f_1,f_2,\ldots,f_n$ are zeroth order annihilators, we have
\[
0 = [P',f_k] \bl \sigma_F = P' \bl (f_k \bl \sigma_F) - f_k \bl (P' \bl \sigma_F) = - f_k \bl (P' \bl \sigma_F). 
\]
Now we put
\[
\left[
\frac{-c(x)}{f_1^{\ell_1}f_2^{\ell_2}\cdots f_n^{\ell_n}} 
\right]
= P' \bl \sigma_F
\]
as a simplified expression.
Here $c(x)\ne 0$.
Then it follows from $f_k \bl (P' \bl \sigma_F)=0$ that $\ell_k = 1$.
It implies
\[
\left[\frac{-c(x)}{f_1f_2\cdots f_n}\right] = P' \bl \sigma_F.
\]
Then $P'+c(x)$ is an annihilator of $\sigma_F$.
\end{proof}

Next, we discuss first order annihilators.
Let $P = \sum_{i=1}^n
a_i(x)\pd_i + c(x)\in \ann_{D_n}(\sigma_F)$.
Since $f_k \in F$ is an annihilator,
a commutator $[P, f_k] = P f_k - f_k P$ is also a zeroth order annihilator.
It implies $\sum_{i=1}^n a_i(x)\frac{\pd f_k}{\pd x_i}\in I$.  In other words, there exist
polynomials $a_i(x), b_{ij}(x)\in K[x]$ such that
\begin{equation}\label{eq:5.1}
\sum_{i=1}^n a_i(x)
\left(\sum_{k=1}^n\frac{\pd f_k}{\pd x_i}\e_k\right)
+
\sum_{i=1}^n
\sum_{j=1}^n
b_{ij}(x)f_i(x)\e_j = \0,
\end{equation}
where $\e_j$ is the $j$-th fundamental vector.
Let
\[
\cM^{(1)} =\left\{ \left.\sum_{k=1}^n\frac{\pd f_k}{\pd x_i}\e_k \, \right| \,
i=1,2,\ldots,n \right\}
\cup
\{ f_i(x)\e_j \mid i,j=1,2,\ldots,n \}.
\]
The solutions $(a_i; b_{ij})$ of the system~$(\ref{eq:5.1})$ form
the syzygy module $\cS$ with respect to $\cM^{(1)}$.
When $(a_i)=\0$, the operator $P=c(x)$ is a zeroth order annihilator.
Therefore we can ignore the case of $(a_i)=\0$.

Suppose $(a_i)\ne\0$.
it follows
from direct calculation that  
$\sum_{i=1}^na_i \pd_i - \sum_{i=1}^n b_{ii} \in \ann_{D_n}(\sigma_F)$.
Notice that $P+p(x)$ is also an annihilator if $p(x)\in I$.
We may choose $c(x) = (- \sum_{i=1}^n b_{ii} \mod I)$ in $L$.
Hence the following algorithm gives a method for computing generators of
$\ann_{D_n}(\sigma_F)$.

\begin{algorithm}[H]
\caption{generators of $\cA^{(1)}$}
\label{alg:annih1st}
\begin{algorithmic}
\REQUIRE $F=\{f_1, \ldots, f_n\}$: a regular sequence
\ENSURE $\cL^{(1)}$: generators of $\cA^{(1)}$.
\STATE $J \gets \left[\frac{\pd f_k}{\pd x_i}\right]_{ki}$: Jacobi matrix
\STATE $E \gets \text{unit matrix of size $n$}$
\STATE $M \gets [ J \, | \, f_1 E \, | \, \cdots \, | \, f_n E ]$: combined matrix
\STATE $S \gets \text{(\groebner basis of syzygy module of columns vectors of $M$)}$
\STATE $\cL^{(1)} \gets F$
\FORALL{$(a_i(x); b_{ij}(x))\in S$}
\IF{$(a_i(x))\ne \0$}
\STATE $c \gets - \sum_{i=1}^n b_{ii} \mod \ag{F}$
\STATE $\cL^{(1)} \gets \cL^{(1)} \cup \{\sum_{i=1}^n a_i(x)\pd_i + c \}$
\ENDIF
\ENDFOR
\RETURN $\cL^{(1)}$
\end{algorithmic}
\end{algorithm}

Next, we discuss $r$-th order annihilators.  
Let $\{P_1,P_2,\ldots,P_q\}$ be a basis of $\cA^{(r-1)}$
and let $P = \sum_{1\le |\alpha| \le r} a_\alpha(x) \pd^{\alpha} + c(x) \in \ann_{D_n}(\sigma_F)$.
If $f\in F$, the commutator $[P,f]=[P-c,f]$ is at most $(r-1)$-th order annihilator.
Then there exist polynomials $b_{k1}, \ldots, b_{kq}$ such that
\[
\sum_{1\le |\alpha| \le r} a_\alpha [\pd^{\alpha},f_k] + \sum_{j=1}^q b_{kj} P_j = 0.\qquad (1\le k\le n)
\]
Using expansions $[\pd^{\alpha},f_k] = \sum_{0\le |\beta| < r} c_{\alpha k\beta}\pd^{\beta}$ and
$P_j = \sum_{0\le |\beta| < r} d_{j \beta}\pd^{\beta}$,
we have 
\[
\sum_{0\le |\beta| < r} 
\left( \sum_{1\le |\alpha| \le r} a_\alpha c_{\alpha k\beta} + \sum_{j=1}^q b_{kj} d_{j \beta}\right)
\pd^{\beta}
 = 0.\qquad (1\le k\le n)
\]
Thus
\[
\sum_{1\le |\alpha| \le r} a_\alpha c_{\alpha k\beta} + \sum_{j=1}^q b_{kj} d_{j \beta}
 = 0.\qquad (1\le k\le n, 0\le |\beta| < r)
\]
Set $M=\binom{n+r-1}{n}$.
Using a bijection $\varphi: \{k \mid 1\le k\le n\}\times \{\beta \mid 0\le |\beta|<r\} \to \{1,2,\ldots,nM\}$,
we can rewrite the system as follows:
\begin{equation}\label{eq:4.2}
\sum_{1\le |\alpha| \le r} a_\alpha
\biggl\{
\sum_{k=1}^n\sum_{0\le |\beta|<r}
c_{\alpha k\beta}\e_{\varphi(k,\beta)}
\biggr\}
+ 
\sum_{j=1}^q \sum_{k=1}^n b_{kj}
\biggl\{
\sum_{0\le |\beta|<r} d_{j \beta}\e_{\varphi(k,\beta)}
\biggr\}
= \0,
\end{equation}
where $\e_{\varphi(k,\beta)}$ are fundamental vectors of dimension $nM$.
Put 
\begin{eqnarray}
&&
\cM^{(r)} =
\left\{
\left.
\sum_{k=1}^n\sum_{0\le |\beta|<r} c_{\alpha k\beta}\e_{\varphi(k,\beta)}
\ 
\right|
\ 
1\le |\alpha| \le r
\right\}
\nonumber
\\
&&
\qquad\qquad
\cup
\left\{
\left.
\sum_{0\le |\beta|<r} d_{j \beta}\e_{\varphi(k,\beta)}
\ 
\right|
\ 
1\le j\le q, 1\le k\le n
\right\}.
\label{eq:4.3}
\end{eqnarray}
Using a \groebner basis of the syzygy module with respect to $\cM^{(r)}$,
we can solve the equation $(\ref{eq:4.2})$.
A solution $(a_\alpha; b_{kj})$ gives the main part $P'
 = \sum_{1\le |\alpha| \le r} a_\alpha(x) \pd^{\alpha}
$ of an $r$-th order annihilator $P$ if there exists $\gamma\in \N_0^n$ such that $a_\gamma \ne 0$ and $|\gamma|=r$.

Finally, we complete the $r$-th order annihilator $P=P'+c(x)$ from the main part $P'$.
Put 
\[
\frac{s(x)}{f_1^{\ell_1}f_2^{\ell_2}\cdots f_n^{\ell_n}} = P' \bl \frac{1}{f_1f_2\cdots f_n}
\]
as an irreducible fraction.  
From Lemma~\ref{lemma:4.1}, as an algebraic local cohomology class, we have 
\[
\left[
\frac{s(x)}{f_1^{\ell_1}f_2^{\ell_2}\cdots f_n^{\ell_n}} 
\right]
=
\Biggl[
\frac{-c(x)}{f_1f_2\cdots f_n}
\Biggr]
=
\left[
\frac{-c(x)f_1^{\ell_1-1}f_2^{\ell_2-1}\cdots f_n^{\ell_n-1}}{f_1^{\ell_1}f_2^{\ell_2}\cdots f_n^{\ell_n}} 
\right].
\]
It implies $s(x) + c(x)f_1^{\ell_1-1}f_2^{\ell_2-1}\cdots f_n^{\ell_n-1} 
\in \ag{f_1^{\ell_1}, f_2^{\ell_2}, \ldots, f_n^{\ell_n}}$.
Using the syzygy method again, the polynomial $c(x)$ can be determined.

\begin{algorithm}[H]
\caption{generators of $\cA^{(r)}$}
\label{alg:annih2}
\begin{algorithmic}
\REQUIRE $r\ge 2$, $F=\{f_1, \ldots, f_n\}$: a regular sequence, $\cL^{(r-1)}$: generators of $\cA^{(r-1)}$
\ENSURE $\cL^{(r)}$: generators of $\cA^{(r)}$.
\STATE $g\gets f_1 f_2 \cdots f_n$
\STATE $\cM^{(r)} \gets$ (the right hand side of the equation $(\ref{eq:4.3})$ by $r$, $F$, $\cL^{(r-1)}$)
\STATE $S \gets$ (\groebner basis of the syzygy module with respect to $\cM^{(r)}$)
\STATE $\cL^{(r)} \gets \cL^{(r-1)}$
\FORALL{$(a_\alpha; b_{kj})\in S$}
\STATE $P' \gets \sum_{1\le |\alpha|\le r} a_\alpha\pd^{\alpha}$
\IF{$\ord(P')=r$}
\STATE $h \gets$ (the result by Lemma~$\ref{lemma:4.1}$)
\STATE $\cL^{(r)} \gets \cL^{(r)} \cup \{P' - h\}$
\ENDIF
\ENDFOR
\RETURN $\cL^{(r)}$
\end{algorithmic}
\end{algorithm}

\begin{example}\label{ex:4.2}
Let $f_1=(x^2-2)(x^4-4x^2-y^4-5y^3-9y^2-7y+2)$ and $f_2=x^4-4x^2-y^3-3y^2-3y+3$.
The sequence $F=\{f_1,f_2\}$ is regular.
Using Algorithm~\ref{alg:annih1st}, a basis of $\cA^{(1)}$ 
of $\sigma_F = [\frac{1}{f_1f_2}]$
can be computed as follows:
\begin{eqnarray*}
&&
f_1, \quad f_2, \quad
3(y+1)^3{\pd_x}+4(y+1)x(x^2-2){\pd_y}+22x(x^2-2),\\
&&
3(x^2-2) {\pd_x}+4(y+1)x {\pd_y}+34 x,
\quad
(y+1)^5{\pd_y}-3(x^2-2)^2+(y+1)^3(7y+10),\\
&&
\{(y+1)^3-(x^2-2)^2\}{\pd_y}+3(y+1)^2
,
\quad
(y+1)(x^2-2)\{(y+1){\pd_y}+4\}
\end{eqnarray*}

Remark that second (or higher) order annihilators are not necessary for computing the Grothendieck local residue mapping in this case.
However they can be calculated 
by using Algorithm~\ref{alg:annih2}
as follows:
\begin{eqnarray*}
&&
\{7(x^2-2)^2-4(y+1)^3\}{\pd_y}^2-66(y+1), 
\quad
(y+1)^3(x^2-2){\pd_y}^2, 
\\
&&
(y+1)^6{\pd_y}^2,
\quad
21(y+1)^2(x^2-2){\pd_x}{\pd_y}+5x(y+1)^3{\pd_y}^2+604x(y+1),
\\
&&
(y+1)^5 {\pd_x}{\pd_y},
\quad
21 (y+1)^4 {\pd_x}^2+16 (y+1)^3{\pd_y}^2.
\end{eqnarray*}
\end{example}

\section{Differential operator $T_{F, \lambda}$ on an irreducible component}\label{sec:6}

As we have discussed, the algebraic local cohomology class $\sigma_F$ can be decomposed to the direct sum
\[
\sigma_F = \sigma_{F,1}+\cdots+\sigma_{F,\lambda}+\cdots+\sigma_{F,\ell}
\]
and each $\sigma_{F, \lambda} \in H_{Z_\lambda}(K[x])$
can be represented as $\sigma_{F, \lambda} = T_{F, \lambda}^* \bl
\delta_{F, \lambda}$.
In this section, we discuss the relation between the differential
operator $T_{F,\lambda}^*$ and Noether space of $I_\lambda$, 
and describe a method for computing $T_{F, \lambda}^*$ without the use
of an explicit representative element of $\sigma_{F,\lambda}$.

From the definition of delta functions, the polynomial ideal $\sqrt{I_\lambda}$ annihilates $\delta_{Z_\lambda}$, that is, 
$h\delta_{Z_\lambda}=0$ for any $h\in \sqrt{I_\lambda}$.
Then we have the following lemma.

\begin{lem}
$\ann_{D_n}(\delta_{Z_\lambda}) = D_n \sqrt{I_\lambda}$.
\end{lem}

If an algebraic local cohomology class $\eta \in H_{[Z_\lambda]}^n(K[x])$ can be represented as
$\eta = U_1^*\bl \delta_{Z_\lambda}$ and $\eta = U_2^*\bl \delta_{Z_\lambda}$
by using two operators $U_1^*, U_2^* \in D_n$, then
$(U_1^* - U_2^*) \bl \delta_{Z_\lambda}=0$.
It implies $U_1^* - U_2^* \in D_n \sqrt{I_\lambda}$.
Hence we should identify $U_1^*$ with $U_2^*$ as an equivalent class in $(D_n \bmod D_n\sqrt{I_\lambda})$.

\begin{algorithm}
\caption{Calculation for $S_\lambda^*$}
\label{alg:4}
\begin{algorithmic}
\REQUIRE $I_\lambda$: a zero-dimensional primary ideal, $G$: a \groebner basis of $\sqrt{I_\lambda}$
\ENSURE $S_\lambda^*$: a differential operator, $r$: maximal order of annihilators.
\STATE $m \gets$ (the multiplicity of $I_\lambda$)
\STATE $d \gets \dim_K K_\lambda$
\STATE $\{R_1=1, R_2, \ldots, R_m\} \gets \text{(the result by Algorithm~\ref{alg:noether3})}$
\STATE $B=\{1,x^{\gamma_2}, \ldots, x^{\gamma_{d}}\} \gets$ (standard monomials with respect to $G$)
\FOR{$i=1$ \TO $m-1$}
\STATE $s_i \gets c_{i, 1}+c_{i, 2}x^{\gamma_2}+\cdots+c_{i, d}x^{\gamma_d}$, where $c_{i,1}, c_{i,2}, \ldots, c_{i,d}$ are symbols
\ENDFOR
\STATE $S_\lambda^* \gets R_m+\sum_{i=1}^{m_\lambda-1} R_i s_i$
\STATE $\cL \gets$ (generators of $\cA^{(1)}$ by Algorithm~\ref{alg:annih1st})
\STATE $\cL'\gets \emptyset$
\STATE $E\gets \emptyset$
\STATE $A \gets \{ \alpha \in \N_0^n \mid |\alpha|\le m\}$
\FOR{$r=1,2,3,\ldots$}
\FORALL{$L \in \cL\setminus\cL'$}
\STATE $\sum_{\alpha\in A} (-\pd)^\alpha u_\alpha \gets L\, S_\lambda^*$
\STATE $E\gets E\cup \{\text{the coefficients of $(u_\alpha\!\mod \sqrt{I_\lambda})$ as a polynomial in $K[B]$} \mid \alpha\in A\}$
\ENDFOR
\STATE Solving the linear system $E$ for $U=(c_{ij})_{ij}$
\IF{no more free variables in $U$}
\STATE Assigning $U$ to $S_\lambda^*$
\RETURN $S_\lambda^*$
\ENDIF
\STATE $\cL'\gets \cL$
\STATE $\cL \gets$ (generators of $\cA^{(r+1)}$ by Algorithm~\ref{alg:annih2})
\ENDFOR
\end{algorithmic}
\end{algorithm}

Let $\varphi \in \Hom_{D_n}(D_n,D_n)$.
Consider the following diagram:
\[
\begin{CD}
\ann_{D_n}(\sigma_{F,\lambda})  @>>> D_n  @>>> D_n\bl \sigma_{F,\lambda}  @>>>  0  \\
@VVV                                @VV{\varphi}V \\
D_n \sqrt{I_\lambda}            @>>> D_n  @>>> D_n\bl \delta_{Z_\lambda} @>>> 0
\end{CD}
\]
If $P\, \varphi(1) \in D_n \sqrt{I_\lambda}$ for any $P\in \ann_{D_n}(\sigma_{F,\lambda})$, 
then we can identify $\varphi$ with an element of $\Hom_{D_n}(D_n\bl \sigma_{F_\lambda},  D_n\bl \delta_{Z_\lambda})$.
Since 
$\Hom_{D_n}(D_n\bl \sigma_{F_\lambda},  D_n\bl \delta_{Z_\lambda}) \subset \Hom_{D_n}(M_{I_\lambda}, M_{\sqrt{I_\lambda}})$
and $D_n\bl \delta_{Z_\lambda} = M_{\sqrt{I_\lambda}}$, 
we call $\varphi$ a {\it Noether operator} associated to $\sigma_{F,\lambda}$.

\begin{lem}[\cite{T2}]\label{lem:5.2}
Let $\varphi \in \Hom_{D_n}(D_n\bl\sigma_{F,\lambda},  D_n \bl\delta_{Z_\lambda})\setminus \{0\}$ and let $S^*$ be the corresponding Noether operator to $\varphi$.
Then, as $K$-vector spaces, we have
\[
\Hom_{D_n}(D_n\bl\sigma_{F,\lambda}, H^n_{[Z_\lambda]}(K[x])) \simeq \myspan_K \{
S^* h \bl \delta_{Z_\lambda} \mid h \in K[x]/\sqrt{I_\lambda}
\}.
\]
\end{lem}

Our target $T_{F,\lambda}^*$ can be regarded as a Noether operator associated to the local cohomology class $\sigma_{F,\lambda}$ and can be written as $T_{F,\lambda}^* = S^* h$.

Let $J_F=\det\left(\frac{\pd(f_1, \ldots, f_n)}{\pd(x_1, \ldots,
x_n)}\right)$ be the Jacobian of the regular sequence $F$
and let $m_\lambda$ be the multiplicity of $I_\lambda$.

\begin{theorem}[\cite{T2, TN2}]
Let $\tau \in H_{[Z_\lambda]}^n(K[x])$.  If
$\ann_{D_n}(\sigma_F)$
annihilates $\tau$ and $J_F \bl \tau =
m_\lambda \delta_{F, \lambda}$, then $\tau=\sigma_{F, \lambda}$ holds.
\end{theorem}

\begin{theorem}[\cite{T2}]
Let $h\in K[x]$.  Then the followings are equivalent.
\begin{enumerate}
\item $\sigma_{F,\lambda} = S^* h \bl \delta_{Z_\lambda}$.
\item $J_F S^* h - m_\lambda \in D_n \sqrt{I_\lambda}$.
\end{enumerate}
\end{theorem}

Next we want to give an explicit algorithm for computing $T_{F,\lambda}^*$.

From the decomposition of $\sigma_F$, it follows $\ann_{D_n}(\sigma_F)\bl\sigma_{F,\lambda} = 0$.
It implies $\ann_{D_n}(\sigma_F)\, T_{F, \lambda}^* \subset D_n\sqrt{I_\lambda}$.
The primary ideal $I_\lambda$ has Noether
operator basis $\cR = \{R_1=1, R_2, \ldots, R_{m_\lambda}\}$,
where $m_\lambda$ is the multiplicity of $I_\lambda$.
It follows from Lemma~\ref{lem:5.2} that 
\[
T_{F,\lambda}^*\bl \delta_{Z_\lambda} \in 
\myspan_K \{R_i K_\lambda \bl \delta_{Z_\lambda} \mid i=1,2,\ldots,m_\lambda\}.
\]
Thus $
T_{F,\lambda}^* \in \bigoplus_{i=1}^{m_\lambda} R_i K_\lambda$.
It can be represented as
\[
T_{F, \lambda}^* \equiv R_{m_\lambda1}\,t_{m_\lambda}(x) + R_{m_\lambda-1}\, t_{m_\lambda-1}(x) + \cdots + R_2\, t_2(x) + t_1(x)  \pmod{D_n\sqrt{I_\lambda}}
\]
where $t_k(x)\in K_\lambda$.
We may suppose $\ord(R_i)\le \ord(R_{m_\lambda})$ for any $i<m_\lambda$ without loss of generality.
Then $t_{m_\lambda}(x)\ne 0$.

Put $S_\lambda^* = T_{F,\lambda}^*\, \cdot\, t^{-1}_{m_\lambda}(x)$.
Then it can be expressed as a ``monic'' operator:
\[
S_\lambda^* \equiv R_{m_\lambda} + R_{m_\lambda-1} s_{m_\lambda-1}(x) + \cdots + R_2 s_2(x) + s_1(x).
\pmod{D_n\sqrt{I_\lambda}}
\]
The operator $S_\lambda^*$ satisfies the equation $\ann_{D_n}(\sigma_F)S_\lambda^* \subset D_n\sqrt{I_\lambda}$.

Let $d_\lambda=\dim_K K_\lambda$ be the degree of field extension $K_\lambda$ over $K$ 
and let $G$ be a \groebner basis of $\sqrt{I_\lambda}$.
Then the finite set $B=\{ x^\gamma \mid \nf_G(x^\gamma) = x^{\gamma}\}$ is a monomial basis of the $K$-vector space $K_\lambda$.
Let $B=\{1,x^{\gamma_2},\ldots,x^{\gamma_{d_\lambda}}\}$.
We utilize the method of undetermined coefficients for finding $s_{i}(x) \in K_\lambda$.
That is, they can be written as
\[
s_i(x) = c_{i,1} + c_{i,2}x^{\gamma_2} + c_{i,3}x^{\gamma_3} + \cdots + c_{i,d_\lambda}x^{\gamma_{d_\lambda}} \in K_\lambda,
\qquad (i=1,2,\ldots,m_\lambda-1)
\]
by using undetermined coefficients
$U_1=\{c_{ij} \mid 1\le i < m_\lambda,\ 1\le j\le d_\lambda\}$.

For any $L\in \ann_{D_n}(\sigma_F)$, the product $LS_\lambda^*$ has an
expansion $LS_\lambda^* = \sum_\alpha \pd^\alpha u_\alpha(L,x)$, where
$u_\alpha(L,x)$ are polynomials written in $s_k(x)$.
From $LS_\lambda^* \equiv 0 \pmod{D_n\sqrt{I_\lambda}}$, it follows
a membership problem
$u_\alpha(L,x)\equiv 0 \pmod{\sqrt{I_\lambda}}$.
By monomial reduction with respect to $G$, the polynomials $u_\alpha(L,x)$ can be expanded as
\[
u_\alpha(L,x) \bmod \sqrt{I_\lambda} = \sum_{i=1}^{d_\lambda} v_{\alpha,i}(L)\, x^{\gamma_i}.
\]
Notice that the coefficients $v_{\alpha,i}(L)$ are affine linear forms in
undetermined coefficients $U_1$.  Accordingly, the equation $LS_\lambda^*
\equiv 0$ can be interpreted as linear equations $v_{\alpha,i}(L) = 0$
for variables $U_1$.  If the annihilating left ideal
$\ann_{D_n}(\sigma_F)$ is generated by $L_1, L_2, \ldots$, then the
relations $L_j S_\lambda^*\equiv 0$, $(j=1,2,\ldots)$ make a system of
linear equations.  Solving the linear system, we can find the values of
the undetermined coefficients.  From this procedure, we have
Algorithm~\ref{alg:4}.

Next, there exists a polynomial $h(x) = t_{m_\lambda}(x)$
such that $T_{F,\lambda}^* \equiv S_\lambda^*\,h(x)$.
We utilize the method of undetermined coefficients for finding $h(x)$ again.
It can be written as
\[
h(x) \equiv c_1 + c_2 x^{\gamma_2} + \cdots + c_{d_\lambda} x^{\gamma_{d_\lambda}} \pmod{\sqrt{I_\lambda}}
\]
by using undetermined coefficients $U_2=\{c_i \mid 1\le i\le d_\lambda\}$.

Put $P=J_F S_{F, \lambda}^* h(x)- m_\lambda$.  The operator can be
represented as $P = \sum_\alpha (-\pd)^\alpha u_\alpha(x)$.
From the theorem above, the operator annihilates the delta function
thus $u_\alpha(x)\equiv 0 \pmod{\sqrt{I_\lambda}}$.
Since the polynomials $u_\alpha(x)$ are affine linear forms in undetermined
coefficients $U_2$, we have a system of linear equations.
Solving the system, we can determine the polynomial $h(x)$.
The procedure can be executed by Algorithm~\ref{alg:5}.

\begin{algorithm}
\caption{Calculation for $T_{F, \lambda}$}
\label{alg:5}
\begin{algorithmic}
\REQUIRE $I_\lambda$: a zero-dimensional primary ideal, $G$: a \groebner basis of $\sqrt{I_\lambda}$
\ENSURE $T_{F,\lambda}^*$: a differential operator, $r$: maximal order of annihilators.
\STATE $m \gets$ (the multiplicity of $I_\lambda$)
\STATE $d \gets \dim_K K_\lambda$
\STATE $B=\{1,x^{\gamma_2}, \ldots, x^{\gamma_{d}}\} \gets$ (standard monomials with respect to $G$)
\STATE $h \gets c_{1}+c_{2}x^{\gamma_2}+\cdots+c_{d}x^{\gamma_d}$, where $c_{1}, c_{2}, \ldots, c_{d}$ are symbols

\STATE $S_\lambda^* \gets \text{(the result of Algorithm~\ref{alg:4})}$
\STATE $J_F \gets \det\left[\frac{\pd f_i}{\pd x_j}\right]_{ij}$: Jacobian
\STATE $J_{F,\lambda} \gets J_F \bmod I_\lambda$
\STATE $A \gets \{ \alpha \in \N_0^n \mid |\alpha|\le m\}$
\STATE $\sum_{\alpha\in A} (-\pd)^\alpha u_\alpha \gets J_{F,\lambda}\, S_\lambda^*\, h - m$
\STATE $E\gets \emptyset$
\FORALL{$\alpha \in A$}
\STATE $u_\alpha \gets u_\alpha \mod \sqrt{I_\lambda}$
\STATE $E\gets E\cup \{\text{coefficients of $u_\alpha$}\}$
\ENDFOR
\STATE Solving the linear system $E$ for $(c_1, c_2, \ldots, c_d)$
\STATE $T_{F, \lambda}^* \gets (\text{Assigning the solution $c_i$ to $S_\lambda^*\, h$})$
\RETURN $T_{F, \lambda}^*$
\end{algorithmic}
\end{algorithm}

\section{Local residue mapping}

As we have explained, the local residue mapping can be represented as
the set $\{(T_{F, \lambda}, Z_\lambda) \mid
\lambda=1, 2, \ldots, \ell\}$.
The irreducible set $Z_\lambda$ is the zero set of the prime ideal $\sqrt{I_\lambda}$.
In section~\ref{sec:6}, we have discussed the method for finding
the differential operator $T_{F, \lambda}^*$ for the algebraic local
cohomology class $\sigma_{F, \lambda}$.
Remark that the computation is parallelizable for each pair $(T_{F, \lambda}, Z_\lambda)$.

\begin{algorithm}[H]
\caption{Local residue mapping for $\sigma_F$}
\label{alg:6}
\begin{algorithmic}
\REQUIRE $F=\{f_1, \ldots, f_n\}$: a regular sequence
\ENSURE $T=\{(T_{F, \lambda}, \sqrt{I_\lambda}) \mid \lambda=1, 2, \ldots, \ell\}$.
\STATE $I_1 \cap \cdots \cap I_\ell \gets (\text{primary decomposition of $I=\ag{F}$})$
\STATE $T \gets \emptyset$
\FOR{$\lambda=1$ \TO $\ell$}
\STATE $T_{F, \lambda}^* \gets \text{(the result of Algorithm~\ref{alg:5})}$
\STATE $T_{F, \lambda}  \gets \text{(the formal adjoint of $T_{F,\lambda}^*$)}$
\STATE $Z_\lambda \gets \sqrt{I_\lambda}$
\STATE $T \gets T \cup \{(T_{F, \lambda}, Z_\lambda)\}$
\ENDFOR
\RETURN $T$
\end{algorithmic}
\end{algorithm}

The algorithm above has been implemented on a computer algebra system Risa/Asir.
First, we give an example of a local residue mapping.

\begin{example}\label{ex:7.1}
For simplicity, we explain a case that a regular sequence gives a primary
ideal.  Let $f_1=(x^2-2)(x^4-4x^2-y^4-5y^3-9y^2-7y+2)$ and
$f_2=x^4-4x^2-y^3-3y^2-3y+3$.  Then the sequence $F=\{f_1,f_2\}$ is
regular and the corresponding ideal $I=\ag{F}$ is primary.  Noether
operators associated to $I$ have been given in Example~\ref{ex:3.6} and
annihilators of $\sigma_F = [\frac{1}{f_1f_2}]$ have been shown in
Example~\ref{ex:4.2}.  
For ease of implementation on a computer algebra system Risa/Asir, 
our program computes formal adjoints of operators in Section~\ref{sec:6}.
Using Algorithm~\ref{alg:4}, the operator $S$ on
$V_\C(\sqrt{J})$ can be calculated as
\begin{eqnarray*}
&&
S=-30 {\pd_x}^4 +150 x {\pd_x}^3-480{\pd_x}^2{\pd_y}^3-135(x^2+3){\pd_x}^2 
+ 720 x {\pd_x}{\pd_y}^3
\\
&&\qquad\qquad
+\frac{1575}{2} x {\pd_x}
-64 {\pd_y}^6-720 {\pd_y}^3-\frac{1575}{2}.
\end{eqnarray*}
The leading coefficient of $T$ can be given by $h(x) = -\frac{1}{184320}x$, thus
\begin{eqnarray*}
&&
T = 
\frac{1}{6144} x {\pd_x}^4-\frac{5}{3072} {\pd_x}^3+\frac{1}{384} x {\pd_x}^2 {\pd_y}^3+\frac{15}{4096} x {\pd_x}^2
\\&&\qquad\qquad
-\frac{1}{128} {\pd_x} {\pd_y}^3-\frac{35}{4096} {\pd_x} +\frac{1}{2880} x {\pd_y}^6+\frac{1}{256} x {\pd_y}^3+\frac{35}{8192} x.
\end{eqnarray*}

\medskip

Next, we show a log of our program.
The procedure \texttt{residuemap} is an implementation of Algorithm~\ref{alg:6}.
\begin{verbatim}
$ asir
[1825] load("oh_alc.rr");
[2452] V=[x,y];
[2453] F=[(x^2-2)*(x^4-4*x^2-y^4-5*y^3-9*y^2-7*y+2),
          x^4-4*x^2-y^3-3*y^2-3*y+3];
[2454] oh_alc.init(V,F);
[2455] oh_alc.residuemap();
[[1/6144*x*dx^4-5/3072*dx^3+(1/384*x*dy^3+15/4096*x)*dx^2
+(-1/128*dy^3-35/4096)*dx+1/2880*x*dy^6+1/256*x*dy^3
+35/8192*x,[x^2-2,y+1]]]
\end{verbatim}

\end{example}

\section{Local residues}

As explained in section~\ref{sec:3}, for any holomorphic function
$\varphi(x)$, the local residue at $\beta \in Z_\lambda$ can be evaluated as
\[
\res_\beta(\varphi \sigma_F dx) = (T_{F,\lambda}\bl \varphi)|_{x=\beta}.
\]
Accordingly, the local residue mapping can be
represented by using a set $\{(T_{F,\lambda},Z_\lambda) \mid
\lambda=1,2,\cdots,\ell\}$.

Given a holomorphic function $\varphi$, it is easy to apply $T_{F,\lambda}$ to the function.
The algorithm for evaluating local residues is as follows.
In the result, the variable $x=(x_1,\ldots,x_n)$ expresses a zero of each component $V_\C(\sqrt{I_\lambda})$.

\begin{algorithm}[H]
\caption{Local residue}
\label{alg:7}
\begin{algorithmic}
\REQUIRE $F=\{f_1, \ldots, f_n\}$: a regular sequence, $\varphi$: a holomorphic function
\ENSURE $\Phi=\{T_{F, \lambda}\varphi \mid \lambda=1, 2, \ldots, \ell\}$.
\STATE $\{(T_{F, \lambda}, \sqrt{I_\lambda}) \mid \lambda=1, 2, \ldots, \ell\} \gets \text{(the result of Algorithm~\ref{alg:6})}$.
\STATE $\Phi \gets \emptyset$
\FOR{$\lambda=1$ \TO $\ell$}
\STATE $\Phi \gets \Phi \cup \{(T_{F, \lambda}\bl \varphi, \sqrt{I_\lambda})\}$
\ENDFOR
\RETURN $\Phi$
\end{algorithmic}
\end{algorithm}

\section{Conclusion}

We have developed an algorithm for exactly evaluating Grothendieck local
residue via a representation of the local residue mapping in the general cases.  
Although our method is based on the theory of algebraic
local cohomology groups and $D$-modules, the algorithm can be realized
by using calculations on polynomial rings.  The resulting algorithms
have been implemented in a computer algebra system Risa/Asir.


\end{document}